\documentclass[11pt,leqno]{article}

\usepackage{amssymb,amsmath,amsthm, url,rotating}
 \usepackage{graphicx,epsfig}
 \usepackage{xcolor,graphicx}
  \usepackage{hyperref}
  \usepackage{mathrsfs}
\newcommand{\n}{\noindent}

\newcommand{\vp}{\varepsilon}
\newcommand{\bb}[1]{\mathbb{#1}}
\newcommand{\cl}[1]{\mathcal{#1}}

\theoremstyle{plain}
\newtheorem{thm}{Theorem}

\newtheorem{cor}[thm]{Corollary}

\theoremstyle{definition}

\theoremstyle{remark}
\newtheorem{rem}[thm]{Remark}

\def\tilde{\widetilde}

\renewcommand{\tilde}{\widetilde}

\def\PP{\bb P}

\def\nl{\nolimits}

\def\PP{\bb P}

\def\tilde{\widetilde}

\renewcommand{\tilde}{\widetilde}

\def\hat{\widehat}
\def\nl{\nolimits}

\font\tenbb=msbm10
\font\fivebb=msbm5
\font\sevenbb=msbm7
\newfam\bbfam
\textfont\bbfam=\tenbb
\scriptfont\bbfam=\sevenbb
\scriptscriptfont\bbfam=\fivebb
\def\bb{\fam\bbfam\tenbb}
\def\vp{\varepsilon}

\def\tilde{\widetilde}

 \def\n{\noindent}

\def\vp{\varepsilon}

\def\PP{{\bb P}}
\def\K{{\cal K}}
\def\tilde{\widetilde}

\def\hat{\widehat}

\begin{document}

\title{On a linearization trick}

\author{by\\
Gilles Pisier \\
Texas A\&M University\\
College Station, TX 77843, U. S. A.}

 \maketitle
\begin{abstract}  
In several situations, mainly involving a self-adjoint set of unitary generators of a $C^*$-algebra, we show that any matrix polynomial in the generators and the unit that is in the open unit ball can be written as a product of matrix polynomials of degree 1 also in the open unit ball.
 \end{abstract}

 	{\bf MSC-class} 46H35, 46L54, 47C15, 15B52, 60B20
	
	
{\bf Author's Email:} pisier@math.tamu.edu

\bigskip

\bigskip

 In random matrix theory, especially in connection with estimates
 of the edge of the spectrum of a random matrix, 
 a certain ``linearization trick" has recently played an important role.
 It was introduced in the Gaussian random
  matrix context by Haagerup and Thorbj\o rnsen
 \cite{HT}, who mention in   \cite{HT}
that they were inspired by a similar trick 
from the author's \cite{P}. The latter can be applied,
 among other settings, to unitary random matrices,
in problems about ``strong convergence" considered more recently 
 by Collins and Male  in \cite{CoM}, and Bordenave and Collins
 in \cite{BoCo}.
 Roughly, one wants to estimate
 the limit of the norm of a ``polynomial" $P(x^{(N)}_1,x^{(N)}_2,...;{x^{(N)}_1}^*,{x^{(N)}_2}^*,... ) $ in large unitary random
$N \times N$-matrices and their inverses when  $N\to \infty$ and
to show that the limit is equal to the norm of the same polynomial
$P(x^{\infty}_1,x^{\infty}_2,...;{x^{\infty}_1}^*,{x^{\infty}_2}^*,... ) $ but with the random matrices replaced
by certain unitary matrices $(x^{\infty}_1,x^{\infty}_2,...) $ that play the role of a limiting object. In such situations,
the main difficulty is to prove  $\lim\nl_{N\to \infty} \|P(x^{(N)}_1,x^{(N)}_2,...) \|\le \|P(x^{\infty}_1,x^{\infty}_2,...)\|$
(say almost surely). By homogeneity, this reduces to
$\|P(x^{\infty}_1,x^{\infty}_2,...)\|<1 \Rightarrow \lim_N \|P(x^{(N)}_1,x^{(N)}_2,...) \|<1$.
Computing the norm of such a polynomial   is usually an intractable problem, but
this   is often more accessible
 for polynomials $P$ of degree 1.
Thus if we had a factorization of any $P$ such that
$\|P(x^{\infty}_1,x^{\infty}_2,...)\|<1$
 as a product 
of polynomials of degree 1 satisfying the same bound, the problem 
would be reduced to a more tractable one. 
While the desired factorization seems hopeless with scalar coefficients,
it turns out to be true
if one allows generalized polynomials with matrices as coefficients,
or equivalently matrices with polynomial entries,
the original polynomial being viewed as a matrix of size 1.
In fact it is more natural
to try to factorize general polynomials with matrix   coefficients
in the open unit ball
as products of polynomials of degree 1 in the same ball.
This is the content of our Theorem \ref{main} below, 
a rather simple factorization of matrix valued polynomials
that seems to be a basic fact, of interest in its own right.
% independently of random matrices.
%possibly on independent interest.

 The ``trick" in  \cite{P}    combines very simply facts and
 ideas 
 commonly used in operator space theory,
  involving completely bounded (or completely  positive) maps (see \cite{ER,Pa,P4}).
  
  The recent survey \cite{Mai}   and the book
  \cite{MiS} 
  mention several areas where an analogous trick
  is known in some form (in some cases going back 50 years), but 
  do not mention the operator space
   %(o.s. in short) 
   connection.
  They describe a linearization 
  due to  Anderson  \cite{An}
  in the form of a factorization of matrices
  with polynomial entries, involving the ``Schur complement".
  However, it turns out that, when combined with ideas due to Blecher and Paulsen
  \cite{BP2},
  the operator space viewpoint also produces a very nice factorization
  theorem that seems to be of independent interest.
 This factorization highlights the fact
 that the operator space structure of the \emph{linear}
 span of the generators
 of an operator algebra
 in many cases determines that
 of the whole operator algebra (see \cite{101} for more on this).
 
  In short, the goal of the present note is to advocate
  the resulting operator space version of the linearization trick.
  % Among other variants,   the main result that follows is  
    
    %, possibly of independent interest.\\
    Throughout this note let $H$ be an arbitrary Hilbert space.
   Let $(x_j)$ be a finite family in the Banach algebra $B(H)$ 
   of all bounded operator on $H$; we denote by $1$ the unit in $B(H)$.
   By a monomial in $(x_j,x_j^*)$ we mean
   a product of terms among the collection $\{1, x_j,x_j^*\}$.
   If the product has at most $d$ terms we say that
   the monomial has degree at most $d$.
   By a polynomial in $(x_j,x_j^*)$ (resp. of degree at most $d$)
   we mean a linear combination of monomials  (resp. of degree at most $d$).
  Let $M_{n,m}$ denote the space of $n \times m$ complex matrices.
  We set as usual $M_{n}=M_{n,n}$.  
  By a (rectangular or square) matrix valued polynomial  (resp. of degree at most $d$)
  in $(x_j,x_j^*)$
 we mean a (rectangular or square) matrix
 with entries that are polynomials in $(x_j,x_j^*)$  (resp. of degree at most $d$).
 The norm of an $n \times m$ matrix valued polynomial
 is the operator norm, i.e.  the norm of the associated matrix in $M_{n,m}(B(H))$.
 
 In its simplest form our main result is as follows:
 \begin{thm}\label{main}
 If the $x_j$'s are all unitary operators, any matrix valued polynomial
  in $(x_j,x_j^*)$ 
with norm $<1$ can be written as a finite product
$P_1P_2  \cdots P_m$
of matrix valued  
 polynomials of degree at most 1 with
$\|P_\ell\|<1$ for all $1\le \ell\le m$.
 \end{thm}
 
 We complete the proof after Remark \ref{r0}.
 
 The statement appearing below as    Corollary \ref{26.8'} 
 is  already in \cite[p. 389]{P4}  (unfortunately the condition on the unit
 is missing there). Theorem \ref{26.8} from which it is deduced
 is implicit there. Both are but  
  a slight generalization of a fundamental factorization result
 due to Blecher and Paulsen \cite{BP2}, itself based  
 on the Blecher-Ruan-Sinclair \cite{BRS}
 characterization of operator algebras. The interest of Theorem \ref{main}
 lies in the fact that it is valid for general unitary operators, in particular
in  the reduced $C^*$-algebra of a group; the 
results of \cite{BP2} are stated for maximal  or universal operator algebras, and 
while one could try a lifting argument to deduce Theorem \ref{main} from them
we do not see how to do this.
 
 For any pair $H_1,H_2$ of Hilbert spaces we denote by $H_1 \otimes_2 H_2$ the Hilbert space
 tensor product.
For any $t\in B(H_1) \otimes B(H_2)$ (algebraic tensor product)
  we   denote simply by $\|t\|_{\min}$,
or more often simply by  $\|t\|$, the norm induced on
$B(H_1) \otimes B(H_2)$  by
$B(H_1 \otimes_2 H_2)$.
By definition, an operator space is a linear subspace
$E\subset B(H)$. Throughout this paper, the space $M_n(E)$ of $n\times n$ matrices with entries in $E$ is always equipped with the norm induced
by $M_n(B(H))=B(H\oplus\cdots\oplus H)$ (with $H$ repeated $n$-times).
We refer to \cite{ER,P4,Pa} for more information on operator space theory.
We just recall that a linear map $u: E_1 \to E_2$ between operator spaces
$E_1\subset B(H_1)$ and $E_2\subset B(H_2)$ is called completely bounded
(c.b. in short) if $\sup\nl_n \|u_n\|<\infty$ where
$u_n: M_n(E_1) \to M_n(E_2)$ is the map taking
 $[a_{ij}]\in M_n(E_1)$ to $[u(a_{ij})]\in M_n(E_2)$, and the corresponding
 norm is defined by
  $\|u\|_{cb}=\sup\nl_n \|u_n\|$.
 
 Let ${\cal
A}\subset B({\cl H})$ be a   unital subalgebra. 
Throughout we identify $M_n({\cal A})$ with $M_n \otimes {\cal A}$. We will identify as usual
$M_n({\cal A})$ with a subset of $M_{n+1}({\cal A})$ 
(by completing a matrix with zero entries).
Then we can think of
$\cup_n M_n({\cal A})$
as a subalgebra of  
$B(\ell_2 ({\cl H}))   $. We equip
$\cup_n M_n({\cal A})$ with its natural operator norm, i.e. the norm induced
on it
 by   $B(\ell_2({\cl H}) )$.\\
 \def\K{{\cl K}}
 For simplicity of notation, we set
 $$\K_0 = \cup_n M_n \subset B(\ell_2),$$
 and we always equip $\K_0 \otimes B({\cl H}) $
 with the norm induced by $B(\ell_2({\cl H}) )$.
 
We will use the identification (as algebras)
 $$ \cup_n M_n({\cal A}) \simeq \K_0 \otimes {\cal A}.$$
 Note  $\K_0 \otimes {\cal A}$ is a subalgebra
 of $B(\ell_2 ({\cl H})) $, generated by
 $(\K_0 \otimes 1_A) \cup (e_{11} \otimes{\cal A})$.\\
We denote by $Id_{E}$ the identity map on a set $E$. \\
  
%This is called the minimal (or spatial) norm.
 
 \begin{thm} \label{26.8} 
 Let $c>0$ be a constant (our main case of interest is $c=1$).\\
 Let ${\cal A}\subset B({\cl H})$ be a unital operator algebra.  
Let ${\cal S}$ be a subset of the unit ball of $ \K_0 \otimes {\cal A}=\cup_n M_n({\cal A})$. We assume that
\begin{equation}\label{j23}
 e_{11} \otimes 1_{\cal A}  \in {\cal S} \end{equation}
  and moreover that 
  $ \K_0 \otimes {\cal A}$  is the algebra generated by $(\K_0 \otimes 1_{\cal A}) \cup{\cal S}  $. \\
Fix an element $x\in \K_0 \otimes {\cal A}$. Then,  the following are equivalent: 

\item{ (i)} For any $H$ and any unital homomorphism $u\colon  {\cal A} \to B(H)$  
  $$\sup\nolimits_{s\in {\cal S}} \|[Id_{\K_0}\otimes u](s)\| \le 1  
  \Rightarrow \|[Id_{\K_0}\otimes u](x)\| <  c.$$

\item{ (ii)} For some $m$ there is  a factorization of the form $x =
\alpha_0D_1\alpha_1\ldots D_m\alpha_m$ where $\alpha_0,\ldots, \alpha_m$
are in $\K_0\otimes 1$ with $\prod_0^m\|\alpha_\ell\|<c$ and where $D_1,\ldots,
D_m$ are elements of $\cup_n M_n({\cl A})=\K_0 \otimes {\cal A}$ represented by block
diagonal matrices of the form
\begin{equation}\label{00} D_{\ell} = \left(\begin{matrix}\boxed{y_1({\ell})}  &&&&\bigcirc\\
& {\boxed{y_2({\ell})}}\\
&&\ddots\\ &&&\ddots\\ \bigcirc&&&&\boxed{y_{N_{\ell}}({\ell})} 
\end{matrix}\right)\end{equation}
with $y_k({\ell})\in {\cal S}$ for all $k$ and ${\ell}$.
\end{thm}
\begin{rem}\label{p26.9} Observe that any $D_{\ell}$ as above is
the product of $N_{\ell}$ factors of the same form but
with all   diagonal coefficients but one equal to $1$. 
Moreover, we can insert  additional $\alpha$ factors in order to rearrange the diagonal terms by means of a conjugation by a permutation matrix. We then
obtain, for a possibly larger length $m$, a  factorization as in (ii) above 
such that whenever  $N_{\ell}>1$ we have $y_2({\ell})=\cdots=y_{N_{\ell}}({\ell})=[1]$ (matrix of size $1\times 1$).\end{rem}
\def\b{{\bullet}}
 \begin{proof}  We start by some preliminaries.
 Let $\cl F$ denote the set of $x\in  \K_0\otimes {\cal A}$
 that admit a factorization $x=\alpha_0D_1\alpha_1\ldots D_m\alpha_m$
 with $\alpha_{\ell} \in \K_0\otimes 1 $ and $D_{\ell}$ as in \eqref{00}.
 We claim that $\cl F=   \K_0\otimes {\cal A} $.
 It is easy to check that 
  if $x,y \in \cl F$   then $  \left( \begin{matrix}   x & 0\\
 0&y
 \end{matrix}\right)$ also belongs to $\cl F$ if $x,y$ admit factorizations with the same $m$. Since we may add 
 diagonal factors with entries equal to $1_{\cl A}$ (which  by \eqref{j23} are of the form \eqref{00}) to equalize the $m$'s if necessary,
 this last condition can always be assumed.
 Moreover, it is obvious that $x\in \cl F$ implies
 $\alpha_0 x\alpha_1 \in \cl F$ for any $\alpha_0, \alpha_1\in \K_0$.
 Therefore, if $x,y \in \cl F$   then $x+y=  ( 1 \ 1 )\left( \begin{matrix}   x & 0\\
 0&y
 \end{matrix}\right) \left( \begin{matrix}   1\\
1
 \end{matrix}\right) \in \cl F$. Now  since the assumption that $\cl S$ and $\K_0\otimes 1$ 
jointly generate  $\K_0\otimes \cl A$   implies
that any $x\in \K_0\otimes \cl A$ is a finite sum of elements of $\cl F$, the claim follows.

We will now equip ${\cal A}$ with an operator space (and actually operator algebra) structure.
We introduce on $\K_0\otimes {\cal A}=\cup M_n({\cal A})$ the norm
$\|x\|_\bullet =\inf \prod_0^m\|\alpha_{\ell}\|$ where the infimum runs over all
factorizations as in (ii). The  preceding claim guarantees that
$\|x\|_\b <\infty$ for any $ x\in \K_0\otimes {\cal A}$. Obviously (using the preceding equalization of the $m$'s)
 \begin{equation}\label{u}
 \forall x,y \in \K_0\otimes {\cal A} \quad
  \left\|   \left( \begin{matrix}   x & 0\\
 0&y
 \end{matrix}\right) \right\|_\b=\max\{ \|x\|_\b ,\|y\|_\b\} \text{  and  }  
  \|xy\|_\b \le \|x\|_\b \|y\|_\b   .
 \end{equation}
For any $x\in M_n\otimes  \cl A=M_n( \cl A)$, let $\|x\|_n= \|x\|_\b$. Then we have 
  \begin{equation}\label{un} \|x\|_{M_n({\cal A})}\le \|x\|_n.\end{equation}
By Ruan's theorem \cite{[Ru1]} (see also \cite{Pa,P4}),
the sequence of norms $(\|. \|_n)$ defines an operator space structure
on $\cl A$. The case $n=1$ defines a norm on $\cl A$
for which by \eqref{un} and our assumption \eqref{j23} on the unit we have
 $\|[1]\|_{1}=\| e_{11} \otimes 1_{{\cal A}}   \|_{\b}=1$.
By \eqref{u}, for any $x,y\in M_n({\cal A})$,  we have
$\| x \odot  y \|_n \le \| x  \|_n  \| y \|_n $ where $\odot$ is the natural product
in the algebra $M_n({\cal A})$, namely
$[x \odot  y]_{ij} =\sum_k x_{ik} y_{kj}$.  
After completion,  
 by the Blecher-Ruan-Sinclair Theorem   \cite{BRS}  (see also \cite{Pa,P4,BLM}),
 $\cl A$ becomes a unital operator algebra $B$ embedded completely isometrically 
 as a unital subalgebra in 
 $B(\mathscr{H})$ for some $\mathscr{H}$ (see also \cite[p. 109]{P4}).
  Let $U: \cl A \to B(\mathscr{H})$ be the resulting unital homomorphism.
  Then  $$ \forall y\in M_n(\cl A) \quad \|y\|_n =\|y\|_\b = \| [Id_{M_n}\otimes U](y)\|_{M_n(B(\mathscr{H}))}.$$
  Equivalently
  $$ \forall y\in \K_0 \otimes \cl A \quad \|y\|_\b   = \| [Id_{\K_0}\otimes U](y)\|.$$
  Let    $s\in \cl S$, 
  obviously $\|s\|_\b\le 1$.
  Therefore $\sup\nl_{s\in \cl S}\|[I_{\K_0}\otimes U](s)\|\le 1$.
  
  Now  let us fix $x$ and assume (i). 
  Then taking $u=U$ we find $\|x\|_\b   = \| [Id_{\K_0}\otimes U](x)\| <c$.
  By definition of $\|. \|_\b$, (ii) follows.
  Thus (i) implies (ii). The converse is obvious.
\end{proof} 
 \begin{cor} \label{26.8'} 
 Let $c\ge 1$ be a constant (our main case of interest is $c=1$).\\
 Let ${\cal
A}\subset B({\cl H})$ be a unital operator algebra.  
Let ${\cal S}$ be a subset of the unit ball of $\cup_n M_n({\cal A})$. We assume \eqref{j23}
  and again  that 
  $ \K_0 \otimes {\cal A}$  is the algebra generated by $(\K_0 \otimes 1_A) \cup{\cal S}  $. Then,  the following are equivalent: 
\item{ (i)} Any unital homomorphism $u\colon \ {\cal A} \to B(H)$ such that
$\sup\nolimits_{x\in {\cal S}} \|[I_{\K_0}\otimes u](x)\| \le  1$ is c.b.\
and satisfies $\|u\|_{cb}\le c$.

\item{ (ii)} For any $n$, any $x$ in $M_n({\cal A})$ with $\|x\|_{M_n(A)}<1$
admits (for some $m=m(n,x)$) a factorization of the form $x =
\alpha_0D_1\alpha_1\ldots D_m\alpha_m$ where $\alpha_0,\ldots, \alpha_m$
are in ${\K_0}\otimes 1$ with $\prod\|\alpha_{\ell}\|<c$ and where $D_1,\ldots,
D_m$ are elements of ${\K_0}\otimes {\cal
A}$
of the form  \eqref{00}.
  \end{cor}

\begin{rem}\label{26.9}
Assume (this is the main case of interest for us) 
that $c=1$, 
and that
$\cl S$ is stable by taking
  block diagonal sums of the form \eqref{00}
with diagonal coefficients in $\cl S$.
Then the factorization in the preceding Corollary \ref{26.8'} can be stated
just like this:\\
Any $x\in M_n(\cl A)$ with $\|x\|<1$ can be written
as a product
\begin{equation}\label{001} x =
\alpha_0D_1\alpha_1\ldots D_m\alpha_m\end{equation} 
with all $D_{\ell}$ in $\cl S$ (of varying sizes)
where the $\alpha_{\ell}$'s are rectangular matrices (of suitable sizes
for the product to make sense, see below)
and   $\|\alpha_{\ell}\|<1$ for all ${\ell}$.
The last point can be adjusted by homogeneity.\\
For the product in \eqref{001} to make sense, we set $N_0=N_{m+1}=n$
and we
implicitly assume that $D_{\ell}$ is of size $N_{\ell} \times N_{\ell}$
and 
$\alpha_{\ell} $  
 of size $N_{\ell} \times N_{{\ell}+1}$.
 Assume $0\in \cl S$ which is harmless. Then  
 we may  add zero entries to the $D_{\ell}$'s  in order to achieve $N_1=\cdots=N_m$.
 Once this is done $\alpha_0 $  and $\alpha_m $  will be the only remaining possibly still rectangular factors.
\end{rem}
\begin{rem}\label{26.10} 
Assume moreover 
that,  whenever it makes sense,  the product $\alpha_0 D \alpha_1$
is  in $\cl S$ for any $D  \in \cl S$  and any pair of matrices $\alpha_0, \alpha_1$
with scalar entries in the open unit ball.
 Then the conclusion can be simplified: any $x\in M_n(\cl A)$ with $\|x\|<1$ can be written
as a product
 \begin{equation}\label{001'} x =
 P_1\ \ldots P_m  \end{equation}
with $P_{\ell} \in \cl S$ for all ${\ell}$.
\end{rem}
\begin{cor}\label{c1}
The factorization described in   \eqref{001} holds in the following cases:
\begin{itemize}
\item[{\rm (i)}] Let $A$ be a unital $C^*$-algebra generated by 
a family of unitaries $(x_j)_{j\ge 1}$.
%We set $u_0=1$ for convenience.
Let $\cl A$ be the unital $*$-algebra generated by $(x_j)_{j\ge 1}$.
Let ${\cal S}$ be the set of
all $x\in \cup M_n({\cal A})$ with $\|x\|\le 1$
of the form either  
 \begin{equation}\label{01} x =a_0\otimes 1+ \sum\nl_{j\ge 1} a_j \otimes x_j\end{equation}
 or
 $$  x= a_0\otimes 1 + \sum\nl_{j\ge 1} a_j \otimes x^*_j$$
where, for some $n$, $j\mapsto a_j$  ($j\ge 0$) is  finitely supported  
with values in $M_n$.
\item[{\rm (ii)}] Let $A$ be a unital $C^*$-algebra generated by 
a   family   $(x_j)_{j\ge 1}$ with only finitely many non-zero elements.
%We set $x_0=1$ for convenience.
Let $\cl A$ be the unital $*$-algebra generated by $(x_j)_{j\ge 1}$.
Let ${\cal S}$ be the set of
all $x\in \cup M_n({\cal A})$ with $\|x\|\le 1$
of the form
 \begin{equation}\label{02}x=a_0\otimes 1+ \sum\nl_{j\ge 1} a_j \otimes x_j+
 \sum\nl_{j\ge 1} b_j \otimes x^*_j+ b \otimes (\sum x^*_jx_j +x_jx^*_j )\end{equation}
where, for some $n$, we have $ a_0,a_j,b_j,b\in M_n$.
\item[{\rm (iii)}] In the same situation as (ii), let ${\cal S}$ be the set of
all $x\in \cup M_n({\cal A})$ such that $x=x^*$ with $\|x\|\le 1$
of the form \eqref{02}.
\item[{\rm (iv)}] In the same situation as (ii), let ${\cal S}$ be the set of
all $x\in \cup M_n({\cal A})$ such that $x=x^*$ with $\|x\|\le 1$
of the form   \begin{equation}\label{3'} x=  a_0  \otimes 1+\sum  a_j  \otimes x_j
+ \sum b_j  \otimes x^*_j + b  \otimes (\sum x^*_jx_j +x_jx^*_j ) ,\end{equation}
where, for some $n$,  we have
   $  a_0,  a_j, b_j ,b\in M_n$   
such that  $ a_0= a_0^*$, $b_j= a_j^*$ for all $j\ge 1$, and  $b=b^*$.
\item[{\rm (v)}] Let $A$ be a unital $C^*$-algebra generated by 
a family of unitaries $(x_j)_{j\ge 1}$.
 Let $\cl A$ be the unital $*$-algebra generated by $(x_j)_{j\ge 1}$.
Let ${\cal S}$ be the set of
all $x\in \cup M_n({\cal A})$ with $\|x\|\le 1$
of the form   \begin{equation}\label{3} x=  a_0  \otimes 1+\sum  a_j  \otimes x_j
+ \sum b_j  \otimes x^*_j,\end{equation}
where, for some $n$,  we have
  $  a_0,  a_j, b_j\in M_n$ 
such that  $a_0=a_0^*$, $b_j=a_j^*$ for all $j\ge 1$ and
$j\mapsto a_j\in M_n$   is finitely supported.

\end{itemize}
 \end{cor}

\begin{proof}
%[First part of the proof of Corollary \ref{c1}]
We first observe that in   case (ii)  the assumption
in Remark \ref{26.9} holds.
As for case (i) we may observe that
any matrix $D$ of the form
$D= \left( \begin{matrix}   y_1 & 0\\
 0&y_2 
 \end{matrix}\right)$ can be written as 
 $D= \left( \begin{matrix}   y_1 & 0\\
 0&1 
 \end{matrix}\right)\left( \begin{matrix} 1 & 0\\
 0&y_2 
 \end{matrix}\right)$, and hence since $1_{M_n}\otimes 1\in \cl S$
 for all $n\ge 1$ we may still factorize with factors in $\cl S$
 even though $\cl S$ contains terms of two types.
\\
(i) We use here the ``linearization trick" from \cite{P}. Let $E={\rm span}[1,\{x_j\mid j\ge 1\}]$. 
Let $u: \cl A\to B(H)$ be a unital homomorphism such that $\sup\nolimits_{s\in {\cal S}} \|[I_{\K_0}\otimes u](s)\| = 1$. We have clearly  $\|u_{|E}\|_{cb}=1$.
A fortiori of course  $\|u(x_j)\|\le 1$ and
since   $x_j$ unitary, we have 
 $u(x_j^*)=u(x_j^{-1}) = u(x_j)^{-1}$, and hence (since $x_j^*\in \cl S$) 
 $\|u(x_j)^{-1}\|\le 1$, so that $u(x_j)$ is unitary for all $j$.
  By Arveson's extension theorem,
 $u$ admits an extension
  $\tilde u: A \to B(H)$ with $\|\tilde u\|_{cb}=1$, and
  $\tilde u(1)=1$ implies that $\tilde u$ is completely positive (c.p. in short),
  see \cite{Pa,P4}.
  Therefore, we have an embedding $H\subset \hat H$
  and a $*$-homomorphism $\pi: A\to B(\hat H)$
  such that $\tilde u (a) = P_H\pi(a)_{|H}$ ($a\in A$).
  Writing $\hat H=H \oplus K$ and
  $$\pi(a)= \left( \begin{matrix} \tilde u (a) & \pi_{12}(a)\\
 \pi_{21}(a)& \pi_{22}(a) 
 \end{matrix}\right)$$
 it is easy to deduce from the fact that $\tilde u (x_j)$ and $\pi (x_j)$ are both unitary that
 $\pi_{12}(x_j)=\pi_{21}(x_j)=0$ for all $j$.
 In other words, 
 $\pi (x_j)$ commutes with $P_H$. Since $\{x_j\}$ generates $\cl A$,
 $H$ is invariant under $\pi(\cl A)$.
 Therefore $\tilde u$   is a  homomorphism
 (and even a $*$-homomorphism)  which must coincide with $u$.
 Thus we conclude $\|u\|_{cb}=1$ and we apply Corollary \ref{26.8'}.\\
 (ii) By decomposing them into real and imaginary parts,
 it is easy to reduce to the case when the $x_j$'s are self-adjoint, so we assume
 that $x_j=x_j^*$ for all $j$.
 Let $E$ be the linear span of $\{1,x_j, \sum x_j^2\}$.
  Let $u: \cl A\to B(H)$ be a unital homomorphism such that $\sup\nolimits_{s\in {\cal S}} \|[I_{\K_0}\otimes u](s)\| = 1$. Again  $\|u_{|E}\|_{cb}=1$, and $u$ admits
  a c.p. extension $\tilde u: \cl A \to B(H)$, which can again be written
   as before as $\tilde u (a) = P_H\pi(a)_{|H}$ ($a\in A$).
   With the same notation as earlier, but now following 
   \cite{HT}, we have for any self-adjoint $a\in E$
    $$\pi(a)= \left( \begin{matrix}   u (a) & \pi_{12}(a)\\
 \pi_{12}(a)^* & \pi_{22}(a) 
 \end{matrix}\right)$$
 and applying that for each $x_j$ as well
 as for $\sum x_j^2$ (on which $\tilde u=u$)
 we  find
  $$\pi(x_j)= \left( \begin{matrix}   u (x_j) & \pi_{12}(x_j)\\
 \pi_{12}(x_j)^* & \pi_{22}(x_j) 
 \end{matrix}\right)$$
 and also
 $\pi(\sum x^2_j)= \left( \begin{matrix}   u (\sum x^2_j) & *\\
 * & * 
 \end{matrix}\right)$.
 But then the equalities $\pi(\sum x^2_j)=\sum \pi(x_j)^2$
 and $u(\sum x^2_j)=\sum u(x_j)^2$ force $\sum \pi_{12}(x_j)\pi_{12}(x_j)^*=0$,
 and hence $ \pi_{12}(x_j)=0$ for all $j$. Again, we conclude that
 $\tilde u$ is a $*$-homomorphism equal to $u$, that $\|u\|_{cb}=1$ and we   apply Corollary \ref{26.8'}.\\
(iii) Let $\cl S_3$ be as in (iii). Let $\cl S_2$ be the corresponding class in (ii).
%Note $\cl S \subset \cl S_2$, and
For any $y\in \cl S_2$ we have
$ \left( \begin{matrix} 0 & y\\
 y^*& 0
 \end{matrix}\right) \in \cl S_3$, and hence $y= \left( 1\ 0\right)\left( \begin{matrix} 0 & y\\
 y^*& 0
 \end{matrix}\right)\left( \begin{matrix} 0  \\
1
 \end{matrix}\right)$. 
 This shows that a factorization   of the form \eqref{001} 
 with $\cl S_2$ can be transformed into one 
 with $\cl S_3$.\\
(iv) Same argument as for (iii).
\\
(v) It is easy to reduce to a \emph{finite} family of unitaries, then
this is a particular case of (iv).
 \end{proof}
 \begin{rem}\label{r0} 
 The preceding argument for (i) shows that the factorization
 \eqref{00} holds even if $\cl S$ is the set
 of $x$'s with $\|x\|\le 1$ of  the form either  \eqref{01} or
 $x=x_j^*$. Indeed, using $x=x_j^*$ suffices to prove that
 $u(x_j)$ is unitary when $\sup\nolimits_{x\in {\cal S}} \|[I_{\K_0}\otimes u](x)\| = 1$.  
 \end{rem}
 
  \begin{proof}[Proof of Theorem \ref{main}]
 Just note that in case (i) (and also  in case (ii))
 we are in the situation described in Remark \ref{26.10}.
  \end{proof}
 
 \begin{rem}
Let $(A_i)_{i\in I}$ be a family of unital 
$C^*$-subalgebras of a unital
$C^*$-algebra $A$. Assume that $\cup_{i\in I} A_i$ generates $A$.
Let ${\cl P}_d$ denote the linear span of all the
  products of $d$ elements in $\cup_{i\in I} A_i$.
  Then any $x\in M_n ({\cl P}_d)$ with $\|x\|<1$ can be written
  as a product $x=P_1  \cdots P_m$
  of (possibly rectangular) matrices
  with entries in ${\cl P}_1$ such that $\|P_j\| <1$
  for all $j$. 
 This follows by  the argument used to prove
 (i) in Corollary \ref{c1} with $\cl S=\cup_{n}M_n({\cl P}_1)$. 
  \end{rem}

 \begin{rem}\label{r1}
 Let $(X_j)$ be a family of non-commuting formal variables (or indeterminates).
 By a $*$-polynomial  $P(X_j,X_j^*)$  
 in     $(X_j)$ we mean
 a linear combination of (non-commuting) products
 (including the empty product denoted by 1)
 of terms taken from $\{ X_j, X_j^*\}$.
 \end{rem}
 Let $A,B$ be  unital $C^*$-algebras.
 Let $(a_j)_{j\in I}$ (resp. $(b_j)_{j\in I}$) be  a family
 in $A$ (resp. $B$).
 We say that $(b_j)$ satisfies the relations satisfied by $(a_j)$
 if, for any $*$-polynomial  $P(X_j,X_j^*)$, the implication
 $P(a_j,a_j^*)=0 \Rightarrow P(b_j,b_j^*)=0$ holds.\\
 When dealing with random matrices, it is formally more general
 to consider the following
 ``almost sure variant": let $(X_j^{N})_{j\in I}$ be a system of random  
matrices of common size $d_N$,  we say that
$(X_j^{N})_{j\in I}$ satisfies a.s.  the relations satisfied by $(a_j)$
if for any $*$-polynomial  $P(X_j,X_j^*)$ such that 
 $P(a_j,a_j^*)=0$  we have $P(X_j^{N},{X_j^{N}}^*)=0$ almost surely.

To illustrate the use of the factorization,
we recover the following known facts (implicit in \cite{P}).

\begin{cor} 
Let $(x_j)_{j\in I}$ be a family of unitary operators in a unital $C^*$-algebra $A$.
Let $(X_j^{N})_{j\in I}$ be a system of random unitary
matrices of common size $d_N$.
We assume that $(X_j^{N})_{j\in I}$   satisfies a.s. the relations satisfied by $(x_j)$
and 
  that for any $n$ and any finitely supported family
$j\mapsto a_j \in M_n$ ($j\in I$) we have
$$\limsup_{N\to \infty} \|\sum a_0\otimes1+a_j \otimes X_j^{N}\|\le \|\sum a_0\otimes1+ a_j \otimes x_j\| \ a.s.$$
 then
 for any $n$, any finite set  $(a_k)$ in $M_n$   and any family of $*$-polynomials $P_k(X_j,X_j^*)$
  we have
 $$\limsup_{N\to \infty} \|\sum a_k \otimes P_k(X_j^{(N)},{X_j^{(N)}}^*  )\|\le \|\sum a_k \otimes P_k(x_j,x_j^*)\| \ a.s.$$
\end{cor}
\begin{proof} 
Let $x=   \sum a_k \otimes P_k(x_j,x_j^*)$
and 
$x^{(N)} =   \sum a_k \otimes P_k(X^{(N)} _j, {X_j^{(N)}}^*  )$. By homogeneity we may
assume $\|x\|< 1$.
By Corollary \ref{c1} we have 
a factorization
$x =
\alpha_0D_1\alpha_1\ldots D_m\alpha_m$
with all factors $D_0,D_1,...$ 
such that either $D$ or $D^*$ is of the form 
$ a_0\otimes1+  \sum a_j \otimes x_j$ with $\|D\|\le 1$
as in Remark \ref{26.9}.
By our assumption on the relations satisfied by $(X_j^{N})_{j\in I}$  (applied to each entry of the matrix $x -
\alpha_0D_1\alpha_1\ldots D_m\alpha_m$)  we have almost surely
$$x^{(N)} =
\alpha_0D^{(N)}_1\alpha_1\ldots D^{(N)}_m\alpha_m$$
where $D^{(N)} _j$ is obtained from $D_j$
by replacing $x_j$ (resp. $x_j^*$) by $X^{(N)} _j$
(resp. ${X^{(N)} _j}^*$) wherever it appears.
This implies
$$\| x^{(N)} \| < (\max_\ell \| D_\ell^{(N)} \| )^m .$$
The conclusion is now immediate.
\end{proof}
\begin{rem} Let  $(x_j)_{j\in I}$ be a family of free Haar unitaries in the sense
of \cite{VDN}. If  a $*$-polynomial satisfies $P(x_j,x_j^*)=0$ then
$P(y_j,y_j^*)=0$ for any family $(y_j)$
of unitaries in a $C^*$-algebra, in particular for any family
 of unitary matrices.
 Thus the assumption on the relations in the preceding
corollary is automatically satisfied if we assume that $(X_j^{N})_{j\in I}$ is formed
of unitary matrices.
\end{rem}

\begin{rem}\label{r2} A similar statement is valid if we replace a.s. convergence by
convergence in probability.
More explicitly,
if we assume that
for any $\vp>0$ and any $a_j$ we have
$$\lim\nl_{N\to\infty} \PP(\{ \|a_0\otimes1+\sum a_j \otimes X_j^{N}\|> \|a_0\otimes1+\sum a_j \otimes x_j\|+\vp\} ) =0$$
then the same argument shows that
for any $\vp>0$, 
 any $n$, any finite set  $(a_k)$ in $M_n$   and any family of $*$-polynomials $P_k(X_j,X_j^*)$
 we have
$$\lim\nl_{N\to\infty} \PP(\{\|\sum a_j \otimes P_j(X^{(N)} _j, {X_j^{(N)}}^*  ) \|> \|\sum a_j \otimes P_j(x_j,x_j^*)\|+\vp\} ) =0   .$$
\end{rem}
\begin{cor} In the situation of the preceding Corollary,
 Assume that for any $n$, any self-adjoint
 $ a_0 \in M_n$ 
and any finite family $(a_j)$ in $M_n$    
  we have
$$\limsup_{N\to \infty} \| a_0\otimes 1 +\sum  a_j \otimes X_j^{N}+\sum  a^*_j \otimes {X_j^{N}}^*
\|\le
  \| a_0\otimes 1 +\sum  a_j \otimes x_j +\sum  a^*_j \otimes {x_j }^*
\|\ a.s.$$
 then for any $n$, any finite set  $(a_k)$ in $M_n$   and any family of $*$-polynomials $P_k(X_j,X_j^*)$
 we have
 $$\limsup_{N\to \infty} \|\sum a_k \otimes P_k(X^{(N)} _j, {X_j^{(N)}}^*  )\|\le \|\sum a_k \otimes P_k(x_j,x_j^* )\| \ a.s.$$
 A similar statement holds for convergence in probability
 as in Remark \ref{r2}.
\end{cor}

\begin{rem} Similar statements hold
for the cases (ii) (iii) (iv) of Corollary \ref{c1}.
This can be applied in particular when $(x_j)$ is a free semi-circular
(or circular) family in the sense of \cite{VDN}.

 \end{rem}

\n{\bf Questions} One major drawback of the method 
to prove factorizations such as \eqref{001} is the lack of
an algorithm allowing one to construct the factors out of the data
that we wish to factorize.
Perhaps a different approach may yield this.\\
Another natural question would be the quest for quantitative estimates of the length of the factorization.
For instance, given a family of unitaries $(x_j)$
(generating a  unital $*$-algebra $\cl A$)
and taking $\cl S$ formed of degree 1 polynomials
as in part (i) or part (v) of Corollary \ref{c1},  
one can ask for estimates (upper and lower)
for the smallest  number $m=m(d,n)$ (resp.  $m=m(d,n,\vp)$ for $\vp>0$ fixed)
satisfying the following: any matricial polynomial
$P\in M_n({\cl A})$  with $\|P\|<1$ of degree at most $d$
can be written as a product $P= P_1\ldots P_m$
of $m$ matricial polynomials of degree at most $1$
with $\|P_\ell\|<1$ for all $\ell$
 (resp.   with   $ \prod_1^m  \|P_\ell\|   <1+\vp$).

\end{document}